\documentclass[12pt]{article}
\usepackage{amssymb,latexsym, amsmath, enumerate, amsthm, mathscinet, mathtools}
\usepackage{color}
\usepackage{hyperref}
\hypersetup{
      CJKbookmarks=true
	colorlinks=true,
	linkcolor=blue,
	filecolor=blue,      
	urlcolor=blue,
	citecolor=blue,
}

\def\R{\mathbb R}
\def\N{\mathbb N}

\def\G{\mathbb G}
\def\H{\mathbb H}
\def\X{\mathrm X}

\def\Z{\mathbb Z}
\def\T{\mathrm T}
\def\S{\mathcal S}
\def\SS{\mathbb S}
\def\M{\mathcal M}

\def\W{\mathcal W}
\def\g{\mathfrak{g}}
\def\L{\mathbb L}
\def\J{\mathbb J}
\def\P{\mathbf P}
\def\bL{\mathbf L}
\def\q{\mathbf q}
\def\BL{\mathrm {BL}}
\def\B{\mathcal B}
\def\CC{\mathbf C}
\def\D{{\mathcal D}^+}
\def\DP{{\mathcal D}^{+,\mathcal{P}}}

\def\RR{\mathbf R}
\def\m{\mathrm m}
\def\h{\hat}
\def\t{\widetilde}
\def\supp{\mathrm{supp}}
\def\dim{\mathrm{dim}}

\def\Var{\mathrm{Var}}
\def\PP{\mathrm{P}}
\def\BV{\mathrm{BV}}
\def\F{\mathcal{F}}
\def\a{\alpha}
\def\aa{\boldsymbol{\alpha}}
\def\ar{\mathbf{a}}
\def\id{\mathrm{id}}

\newtheorem{theorem}{Theorem}
\newtheorem{lemma}{Lemma}
\newtheorem{remark}{Remark}
\newtheorem{corollary}{Corollary}
\newtheorem{proposition}{Proposition}

\title{Loomis--Whitney inequalities on corank $1$ Carnot groups}
\author{Ye Zhang}
\date{}

\begin{document}

\renewcommand{\theequation}{\thesection.\arabic{equation}}
 \maketitle

\vspace{-1.0cm}

\bigskip
	
	{\bf Abstract.} In this paper we provide another way to deduce the Loomis--Whitney inequality on higher dimensional Heisenberg groups $\H^n$ based on the one on the first Heisenberg group $\H^1$ and the known nonlinear Loomis--Whitney inequality (which has more projections than ours). Moreover, we generalize the result to the case of corank $1$ Carnot groups and products of such groups. Our main tool is the modified equivalence between the Brascamp--Lieb inequality and the subadditivity of the entropy developed in \cite{CC09}. 
	
	\medskip
	
	
	\medskip
	


\bigskip

\section{Introduction}
\setcounter{equation}{0}

\subsection{Loomis--Whitney and Brascamp--Lieb inequalities}\label{ss11}

On Euclidean space $\R^k$, $k \in \N^* = \{1,2,\ldots\}$, recall the {\it  Loomis--Whitney inequality on $\R^k$} is the following geometric inequality:
\begin{align}\label{LWc}
\m_k(E) \le \prod_{j = 1}^k \m_{k - 1}(\P_j(E))^{\frac{1}{k - 1}}, \qquad \forall \, E \ \mbox{measurable}.
\end{align}

Here $\m_k$ is the Lebesgue measure on $\R^k$ and for $1 \le j \le k$, the projection $\P_j : \R^k \to \R^{k  -1}$ is defined by $\P_j (x) = \h{x}_j$, where $\h{x}_j$ denotes the point in $\R^{k -1}$ obtained by simply deleting the $j$-th coordinate of $x \in \R^{k}$.

The original proof of the Loomis--Whitney inequality dates back to \cite{LW49} by a discrete argument. It is one of the most important inequalities in mathematics and has applications to not only Sobolev inequalities and embedding \cite{AF03, S02} but also multilinear Kakeya inequality \cite{G15} (see also \cite{BCT06}). For more details and applications of the Loomis--Whitney inequality, we refer to \cite{BN03, BZ88, CGG16, F92} and the references therein. 

The Loomis--Whitney inequality has a far-reaching generalization, called the Brascamp--Lieb inequality, which also generalizes the classical H\"older and Young's inequalities. It was first formulated in \cite{BL76} to find the best constants in Young's inequality.

In general,  the {\it  Brascamp--Lieb inequality} has the following form
\begin{align}\label{BL}
\int_{\R^k} \prod_{j  =1}^m f_j^{q_j} (L_j (x)) dx \le \BL(\bL,\q) \prod_{j = 1}^{m} \left( \int_{\R^{k_j}} f_j(t) dt \right)^{q_j}, 
\end{align}
for all non-negative measurable functions $f_j$ on $\R^{k_j}$, $1 \le j  \le m$.
Here $k, m \in \N^*$, and for every $1 \le j \le m$, $q_j \ge 0$, $k_j \in \N^*$, and $L_j : \R^k \to \R^{k_j}$ is a linear surjection. Furthermore, $\bL := (L_1, \ldots, L_m)$, $\q := (q_1, \ldots, q_m)$, and $(\bL, \q)$ is called the {\it Brascamp--Lieb datum}. The {\it Brascamp--Lieb constant} $\BL(\bL,\q)$ denotes the smallest constant for which \eqref{BL} holds and it could be $+\infty$.

A fundamental theorem in the study of the Brascamp--Lieb inequality is the  Lieb's theorem (see \cite[Theorem 6.2]{L90}) which shows that the Brascamp--Lieb constant is exhausted by centered Gaussian functions. As a result, in \cite{BCCT08} the authors showed the following theorem.

\begin{theorem}[Theorem 1.13 and Proposition 2.8 of \cite{BCCT08}, see also Theorem 6 of \cite{B982}]\label{tfBL}
Let $(\bL,\q)$ be a Brascamp--Lieb datum. Then the Brascamp--Lieb constant $\BL(\bL,\q)$ is finite if and only if we have the scaling condition
\begin{align}\label{scaling}
 k = \sum_{j = 1}^m q_j k_j,
\end{align}
and the dimension condition 
\begin{align}\label{dim}
\dim(V) \le \sum_{j = 1}^m q_j \dim(L_j V), \qquad \forall   \ \mbox{subspace $V \subset \R^k$}.
\end{align}
In particular, we have $\BL(\bL,\q) = 1$ if the following geometric condition holds
\begin{align}\label{geoc}
L_j L_j^* = \id_{k_j}, \quad \forall \, 1 \le j \le m, \quad  \sum_{j = 1}^m q_j L_j^* L_j = \id_k.
\end{align}
Here $\id_k$ denotes the identity matrix on $\R^k$.
\end{theorem}

\begin{remark}\label{rset}
If the geometric condition \eqref{geoc} holds, we call $(\bL,\q)$  {\it geometric Brascamp--Lieb datum} and \eqref{BL} {\it geometric Brascamp--Lieb inequality}. In particular, if $m = k$ with $k \ge 2$ and for $1 \le j \le k$, $q_j = \frac{1}{k - 1}$, $k_j = k - 1$, and $L_j = \P_j$, then we have
\begin{align}\label{LWf}
\int_{\R^k} \prod_{j  =1}^k f_j^{\frac{1}{k - 1}}(\P_j(x)) dx \le \prod_{j = 1}^k \left( \int_{\R^{k - 1}} f_j(t) dt\right)^{\frac{1}{ k - 1}}
\end{align}
holds for all non-negative measurable functions $f_1, \ldots, f_k$ on $\R^{k - 1}$. Furthermore, for every measurable set $E$, if we choose $f_j = \chi_{\P_j(E)}$, since $E \subset \cap_{j = 1}^k \P_j^{-1}( \P_j(E) )$, we have $\chi_{E} \le \prod_{j = 1}^k \chi_{\P_j(E)} \circ \P_j = \prod_{j = 1}^k \chi_{\P_j(E)}^{\frac{1}{k - 1}} \circ \P_j $ and thus \eqref{LWf} implies the Loomis--Whitney inequality \eqref{LWc}. As a result, \eqref{LWf} is also called  Loomis--Whitney inequality (for functions).
\end{remark}

In the literature, there are several approaches to prove Loomis--Whitney and Brascamp--Lieb inequalities, such as using rearrangement inequality \cite{BL76, BLL74}, optimal transport \cite{B982, B98, CGS19}, heat flow monotonicity \cite{BCCT08, CLL04}, and entropy \cite{CC09}.

The target of this paper is to establish (acutally re-establish for the case of Heisenberg groups) the Loomis--Whitney inequality (for functions) on corank $1$ Carnot groups by the method of entropy. As far as the author knows, due to the non-commutative nature of the underlying group, to apply the  heat flow monotonicity approach is not an easy task. 

\subsection{Corank $1$ Carnot groups}\label{ss21}

A  {\it Carnot group} is a connected and simply connected Lie group $\G$ whose Lie algebra $\g$ 
has a stratification $\g = \bigoplus_{j=1}^s \g_j$, that is, a linear splitting $\g = \bigoplus_{j=1}^s \g_j$ where $[\g_1,\g_j]=\g_{j+1}$ for $j =1,\dots,s-1$ and $[\g_1,\g_s]=\{0\}$. If $\g_s \ne \{0\}$, the number $s$ is called the {\it step} of $\G$. The {\it homogeneous dimension} of $\G$ is given by $Q:= \sum_{j = 1}^s j \mathrm{dim} \g_j$.

We call a Carnot group a {\it corank $1$ Carnot group} if the step $s = 2$ and $\mathrm{dim} \g_2 = 1$. Corank $1$ Carnot groups are generalizations of Heisenberg groups and they may admit nontrivial abnormal geodesics. As a result, many topics are studied on corank $1$ Carnot groups (cf. \cite{BKS19, LZ19, R16}). For a complete characterization of the geodesics as well as cut loci on corank $1$ Carnot groups, we refer to \cite{BKS19, R16}. We will see in the following that the Loomis--Whitney inequalities on corank $1$ Carnot groups display different features from the ones on Heisenberg groups. See Remark \ref{noniso} below for more details.

By using the (group) exponential map, we can always identify a Carnot group $\G$ with its Lie algebra $\g$ (cf. \cite[Chapter 3]{BLU07}). 
More precisely, after choosing a suitable basis on $\g$, the corank $1$ Carnot group $\G$ can be identified with $\R^{d + 2n + 1} \cong \R^{d + 2n} \times \R$ with the group structure 
\[
(x,t) \cdot (x',t') = \left(x + x', t + t' + \frac{1}{2} \sum_{j = 1}^n \alpha_j (x_{d + 2j - 1}x_{d + 2j}' - x_{d + 2j} x_{d + 2j - 1}' ) \right).
\]
Here $d \in \N = \{0,1,\ldots\}$, $n \in \N^* = \{1,2,\ldots\}$, and 
\[
0 < \a_1 \le \ldots \le \a_n < +\infty.
\]
We refer to \cite[Section 3]{R16} or \cite[Section 1]{BKS19} for more details about this identification.
In the following we denote the group $\G$ by $\H(d,\aa) $ with $\aa := (\a_1,\ldots,\a_n)$. When $d = 0$ and $\aa = (1,\ldots,1) \in \R^n$, $\H(d,\aa) $ is just the usual $n$-th Heisenberg group and we denote it by $\H^n$.

A canonical basis of $\g_1$ of $\H(d,\aa)$ is given by the following  left-invariant vector fields:
\[
\X_j := 
\frac{\partial}{\partial x_j}, \qquad \forall \, 1 \le j \le d
\]
and 
\[
\X_{d + 2j - 1} := \frac{\partial}{\partial x_{d + 2j -1}} - \frac{\alpha_j }{2} x_{d + 2j} \frac{\partial}{\partial t}, \quad \X_{d + 2j } := \frac{\partial}{\partial x_{d + 2j}} + \frac{\alpha_j }{2}  x_{d + 2j - 1} \frac{\partial}{\partial t}, \quad  \forall \, 1 \le j \le n.
\]
The basis of $\g_2$ is just given by $\T := \frac{\partial}{\partial t}$. Note that the nontrivial bracket relations on $\g$ are $[\X_{d + 2j - 1}, \X_{d + 2j }] = \alpha_j \T$, $1 \le j \le n$.

The {\it horizontal gradient} and {\it canonical sub-Laplacian} on $\H(d,\aa)$ are given respectively by
\[
\nabla := (\X_1, \ldots, \X_{d + 2n}), \qquad \mbox{and} \qquad \Delta := \sum_{j = 1}^{d + 2n} \X_j^2.
\]

Following \cite{FP22}, we give the definition of our (nonlinear) projections on $\H(d,\aa)$. To this end, for $1 \le j \le d + 2n +1$, we define the subgroup of $\H(d,\aa)$ by $\L_j := \R e_j$, where $e_j$ denotes the vector in $\R^{d + 2n + 1}$ with the $j$-th coordinate $1$ and the other coordinates $0$. Now define 
\[
\J_j := \{(x,t) \in \H(d,\aa) : \, x_j = 0\}, \quad \forall \, 1 \le j \le d + 2n, \qquad 
\J_{d + 2n + 1} = \{(x,0) \in \H(d,\aa)\}.
\]
Now fix a $j \in \{1, \ldots, d + 2n + 1\}$. It is easy to see that for every $(x,t) \in \H(d,\aa)$, there is a unique decomposition $(x,t) = (y,s) \cdot \ell e_j$ with $(y,s) \in \J_j$ and $\ell e_j \in \L_j$. There is a natural way to identify $\J_j$ with $\R^{d + 2n}$ by deleting the $0$ on the $j$-th coordinate. So we just define the {\it $j$-th projection on $\H(d,\aa)$}, $\pi_j: \H(d,\aa) \cong \R^{d + 2n + 1} \to \R^{d + 2n}$, by first finding the unique element in $\J_j$ from the  decomposition above and then identifying with an element in $\R^{d + 2n}$.

To be more precise, we can write them down explicitly:
\begin{align}\label{pro1}
\pi_j(x,t) &= (\h{x}_j,t),  \qquad \forall \, 1 \le j \le d, \qquad \pi_{d + 2n + 1}(x,t) = x, \\
\label{pro2}
\pi_{d + 2j - 1}(x,t) &= \left(\h{x}_{d + 2j - 1}, t + \frac{\alpha_j}{2} x_{d + 2j - 1} x_{d + 2j}\right), \\\label{pro3}
\pi_{d + 2j}(x,t)   &= \left(\h{x}_{d + 2j}, t - \frac{\alpha_j}{2} x_{d + 2j - 1} x_{d + 2j}\right), \quad \forall \, 1 \le j \le n.
\end{align}
Recall that $\h{x}_j$ denotes the point in $\R^{d + 2n -1}$ obtained by simply deleting the $j$-th coordinate of $x \in \R^{d + 2n}$.

\begin{remark}\label{Rie}
Different from \cite{FP22}, we also introduce the extra projection $\pi_{d + 2n + 1}$ since we will use the nonlinear Loomis--Whitney inequalities in our proof of the main theorem, which requires more projections. See Proposition \ref{pWL} (as well as the discussion before it) below for more details.
\end{remark}

From definition, for every $j \in \{1,\ldots,d + 2n + 1\}$ it is easy to see that $(x,t)$ and $(x,t) \cdot \ell e_j$ have the same $j$-th projection on $\H(d,\aa)$. 

\begin{lemma}\label{tri}
On the corank $1$ Carnot group $\H(d,\aa)$, for any $1 \le j \le d + 2n + 1$, we have 
\[
\pi_j(x,t) = \pi_j((x,t) \cdot \ell e_j), \qquad \forall \, (x,t) \in \H(d,\aa), \ell \in \R.
\]
In particular, a function $F$ on $\H(d,\aa)$ is invariant under $\L_j$ in the sense that 
\begin{align}\label{invar}
F(x,t) = F((x,t) \cdot \ell e_j), \qquad \forall \, (x,t) \in \H(d,\aa), \ell \in \R
\end{align}
if and only if there exists a function $\t{F}$ on $\R^{d + 2n}$ such that $F = \t{F} \circ \pi_j$.
\end{lemma}

\begin{remark}\label{rel}
By \cite[Lemma 1.5.4]{B18}, for smooth function $F$, \eqref{invar} is equivalent to say $\X_j F = 0$ if $1 \le j \le d + 2n$ and $\T F = 0$ if $j  = d + 2n + 1$. In fact, in the proof of \cite[Theorem 3.8]{KLZZ23}, the auxiliary function can be constructed by using the projection on $\H^1$ w.r.t. the subgroup $\R (a,b,0)$. Actually, this idea can be further generalized to give similar results on general Carnot groups. 
\end{remark}

Finally we define the {\it dilation} on $\H(d,\aa)$ by 
\begin{align}\label{defdi}
\delta_r(x,t) = (rx, r^2 t), \qquad \forall \, r > 0, (x,t) \in \H(d,\aa).
\end{align}

For every $1 \le j \le d + 2n + 1$, $\J_j$ also admits a dilation structure inherited from $\H(d,\aa)$. After identifying with $\R^{d + 2n}$, we define the {\it $j$-th dilation structure} on $\R^{d + 2n}$, denoted by $\delta^{(j)}_r$, by requiring the following equation holds:
\begin{align}\label{dlc}
\delta^{(j)}_r \circ \pi_j = \pi_j \circ \delta_r, \qquad \forall \, r > 0, 1 \le j \le d + 2n + 1. 
\end{align}

\subsection{Main result}\label{ss13}

Now we can state our main theorem of this paper. In the following we use $\|f\|_p$ to denote the $L^p$ norm of the function $f$.

\begin{theorem}\label{t1}
On corank $1$ Carnot group $\H(d,\aa)$, it holds that
\begin{align}\label{LWco1}
\int_{\H(d,\aa)} \prod_{j = 1}^{d + 2n} f_j(\pi_j(x,t)) dxdt \le \CC(d,\aa) \prod_{j = 1}^d \|f_j\|_{d + 2n + 1} \prod_{j = d + 1}^{d + 2n} \|f_j\|_{\frac{n(d + 2n + 1)}{n + 1}},
\end{align}
for all non-negative measurable functions $f_1, \ldots, f_{d + 2n}$ on $\R^{d + 2n}$, where 
\begin{align}\label{defCC}
\CC(d,\aa) := \frac{\|\RR\|_{\frac{3}{2} \to 3}^{\frac{3}{d + 2n + 1}}}{\left( \prod_{j = 1}^n \a_j\right)^{\frac{1}{n(d + 2n + 1)}}}.
\end{align}
Here $\|\RR\|_{\frac{3}{2} \to 3} < +\infty$ denotes the operator norm of the Radon transform $\RR$ from $L^{\frac{3}{2}}(\R^2)$ to $L^3(\SS^1 \times \R)$.
\end{theorem}

In \cite[Section 5]{FP22}, the authors established Theorem \ref{t1} when $d = 0$. Although the constant is not written explicitly, it can be obtained if we track the constants from the argument there carefully.

To be more precise, they first converted the problem of the Loomis--Whitney inequality on the first Heisenberg group $\H^1$ to the  one of the boundedness from $L^{\frac{3}{2}}(\R^2)$ to $L^3(\SS^1 \times \R)$ of the {\it Radon transfrom (or X-ray transform)} $\RR$ defined by
\[
\RR f(\sigma,s) := \int_{\langle x , \sigma \rangle = s} f(x) dx, \qquad \forall \, \sigma \in \SS^1, s \in \R.
\]
Here $\SS^1 := \{x \in \R^2 : \, |x| = 1\}$ is the unit circle on $\R^2$, $\langle \cdot, \cdot \rangle$ denotes the inner product on $\R^2$, and $dx$ is the $1$-dimensional Lebesgue measure on the line $\{x \in \R^2 : \, \langle x , \sigma \rangle = s\}$. See \cite{OS82} for more details about the boundedness of the Radon transfrom $\RR$.

Then their result for the Loomis--Whitney inequality on the first Heisenberg group $\H^1$ is stated as follows. Recall that we usually use $(x,y,t)$ to denote a point in the first Heisenberg group $\H^1$ and by \eqref{pro2}-\eqref{pro3}, the two projections are defined by
\[
\pi_1(x,y,t) = \left(y, t + \frac{1}{2} xy\right), \qquad \pi_2(x,y,t) = \left(x, t - \frac{1}{2}xy\right).
\]

\begin{theorem}[Theorem 2.4 of \cite{FP22}] \label{tH1}
On the first Heisenberg group $\H^1$, it holds that
\begin{align}\label{LWH1}
\int_{\H^1} f_1\left(y, t + \frac{1}{2} xy\right) f_2 \left(x, t - \frac{1}{2}xy\right)  dxdydt \le  \|\RR\|_{\frac{3}{2} \to 3} \|f_1\|_{\frac{3}{2}} \|f_2\|_{\frac{3}{2}},
\end{align}
for all non-negative measurable functions $f_1, f_2$ on $\R^{2}$.
\end{theorem}

For higher dimensional Heisenberg groups $\H^n \, (n \ge 2)$, the authors in \cite{FP22} used the induction argument on the estimates corresponding to the extreme points of the Newton polytope defined in \cite[Section 3]{S11} and applied the multilinear interpolation to obtain the Loomis--Whitney inequality $\H^n$. This argument can be modified to the general $\aa$ case with $d = 0$. See \cite[Section 5]{FP22} for more details. Recently it is generalized one step further to Brascamp--Lieb type inequalities on Heisenberg groups in \cite{HS24}.

\begin{remark}\label{noniso}
On the right-hand side of \eqref{LWco1} in Theorem \ref{t1}, the exponent (of the $L^p$ space) for the first $d$ terms  is different from the exponent for the other terms. In fact, it is because $\X_1, \ldots, \X_d$ commute with all the vector fields and thus will not generate any nontrivial element in $\g_2$. More precisely, from \cite[Theorem 2]{S11}, we cannot expect any estimate of which the corresponding point lies outside of the Newton polytope defined in \cite[Section 3]{S11}. Furthermore, the inequality should also be invariant under dilation, that is, replacing $f_j$ by $f_j \circ \delta_r^{(j)}$, we will obtain the same inequality. In general, these two restrictions will not allow the same exponent for every term on corank $1$ Carnot group $\H(d,\aa)$ with $d \in \N^*$. 

For example, on the simplest example $\H(1,1) \cong \R \times \H^1$. If there is a $p$ such that 
\[
\int_{\R \times \H^1} \prod_{j = 1}^3 f_j \circ \pi_j \le C \prod_{j = 1}^3 \|f_j\|_{p}
\]
for some $C > 0$. By dilation invariance we get $p = \frac{12}{5}$ but the corresponding point $(\frac{5}{3},\frac{5}{3},\frac{5}{3})$ lies outside of the Newton polytope  $[1,+\infty) \times [2,+\infty) \times [2,+\infty)$. On the other hand, our inequality \eqref{LWco1} always corresponds to a point on the boundary of the Newton polytope and thus \cite[Theorem 3]{S11} cannot be applied.

This phenomenon implies that even at the level of Loomis--Whitney inequality on general Carnot group, the situation is more difficult than we expected and we should take the Lie bracket generating relations into consideration. See also Theorem \ref{tpro} for more examples.
\end{remark}



\subsection{Nonlinear Brascamp--Lieb inequalities}\label{ss12}

Since our $\pi_j$ is nonlinear for $d + 1 \le j \le d + 2n$, it is natural to resort to the known results for nonlinear Loomis--Whitney or Brascamp--Lieb inequalities. In fact, by the induction-on-scales argument, in \cite{BBBCF20} the authors obtained the following nonlinear variant of Brascamp--Lieb inequality. See also  \cite{BCW05} for the case of nonlinear Loomis--Whitney inequalities.

\begin{theorem}[Theorem 1.1 of \cite{BBBCF20}]\label{tnBL}
Let $(\bL,\q)$ be a Brascamp--Lieb datum and suppose that $B_j : \R^k \to \R^{k_j}$ are $C^2$ submersions in a neighborhood of a point $x_0$ and $d B_j(x_0) = L_j$ for $1 \le j \le m$. Then for every $\epsilon > 0$ there exists a neighborhood $U$ of $x_0$ such that 
\begin{align}\label{locBL}
\int_U   \prod_{j  =1}^m f_j^{q_j} (B_j (x)) dx  \le (1 + \epsilon) \BL(\bL,\q) \prod_{j = 1}^{m} \left( \int_{\R^{k_j}} f_j(t) dt \right)^{q_j}, 
\end{align}
holds for all non-negative measurable functions $f_j$ on $\R^{k_j}$, $j = 1, \ldots, m$.
\end{theorem}

However, our situation is more or less like the multilinear Radon-like transforms in  \cite{S11, TW03} with no finite Brascamp--Lieb constant. To be more precise, note that for every $1 \le j \le d + 2n$, we have $d \pi_j (0) = \P_j$. However, considering the $1$-dimensional subspace $\{\L_j\}_{j = 1}^{d + 2n + 1}$, by the two conditions in Theorem \ref{tfBL}, we cannot find any Brascamp--Lieb datum with finite Brascamp--Lieb constant. Thus, to apply Theorem \ref{tnBL}, we have to add the extra projection $\pi_{d + 2n +1}$. Another difficulty is that Theorem \ref{tnBL} is local in nature because of the  appearance of the neighborhood $U$ in \eqref{locBL}. Fortunately, this can also be overcomed owing to the dilation structure of $\H(d,\aa)$. In fact, we have the following nonlinear Loomis--Whitney inequality on $\H(d,\aa)$.

\begin{proposition}\label{pWL}
On corank $1$ Carnot group $\H(d,\aa)$, it holds that
\begin{align}\label{pLWco1}
\int_{\H(d,\aa)} \prod_{j = 1}^{d + 2n + 1} f_j(\pi_j(x,t)) dxdt \le \prod_{j = 1}^{d + 2n + 1} \|f_j\|_{d + 2n },
\end{align}
for all non-negative measurable functions $f_1, \ldots, f_{d + 2n + 1}$ on $\R^{d + 2n}$.
\end{proposition}

\begin{proof}
Since for every $1 \le j \le d + 2n + 1$, we have $d \pi_j (0) = \P_j$ and it is easy to check that 
\[
\P_j\P_j^* = \id_{d + 2n}, \qquad  \frac{1}{d + 2n} \sum_{j = 1}^{d + 2n + 1} \P_j^* \P_j = \id_{d + 2n + 1}.
\]
Thus, by Theorem \ref{tfBL}, it is the case of geometric Brascamp--Lieb inequality and the Brascamp--Lieb constant is $1$. Thus, it follows from Theorem \ref{tnBL} that for every $\epsilon > 0$, there exists a neighborhood $U$ such that (with $f_j^{d + 2n}$ replacing the original $f_j$)
\[
\int_U   \prod_{j  =1}^{d + 2n + 1} f_j (\pi_j (x,t)) dxdt  \le (1 + \epsilon) \prod_{j = 1}^{d + 2n + 1}\| f_j \|_{d + 2n}, \qquad \forall \, f_j \ \mbox{non-negative, measurable}.
\]
Now for every $r > 0$, we use $\{f_j \circ \delta_r^{(j)}\}_{j = 1}^{d + 2n + 1}$ to replace $\{f_j\}_{j = 1}^{d + 2n + 1}$ and obtain
\[
\int_{\delta_r(U)}   \prod_{j  =1}^{d + 2n + 1} f_j (\pi_j (x,t)) dxdt  \le (1 + \epsilon) \prod_{j = 1}^{d + 2n + 1}\| f_j \|_{d + 2n}, \qquad \forall \, f_j \ \mbox{non-negative, measurable}.
\]
by \eqref{dlc} after a change of variables. Then letting $r \to +\infty$ first and then $\epsilon \to 0^+$, we obtain \eqref{pLWco1}.
\end{proof}

For other nonlinear results like multilinear Radon-like transforms in  \cite{S11, TW03}, as we mentioned before in Remark \ref{noniso} that our inequality \eqref{LWco1} always corresponding to a point on the boundary of the Newton polytope and thus \cite[Theorem 3]{S11} cannot be applied.

\subsection{Idea of the proof}\label{ss15}

In this article we will give the proof of Theorem \ref{t1} based on Theorem \ref{tH1} and Proposition \ref{pWL}. In other words, we find another way to deduce the Loomis--Whitney inequality on $\H(0,\aa)$ from the one on $\H^1$ other than the approach given in \cite{FP22} and generalize it to the case of $\H(d,\aa)$ with $d \in \N^*$. 

To be more precise, our main tool is the equivalence between the Brascamp--Lieb inequality and the subadditivity of the entropy (see Section \ref{s2} for more details). It is worthwhile to mention that from the entropy point of view, the multilinear interpolation argument in \cite{FP22} becomes more transparent. This idea is originally from \cite{CC09} and we can further show that in our case, we only need to know the subadditivity of the entropy  on a suitable class where basic entropy operations (such as the conditional entropy and pushforward entropy) are feasible. See Theorem \ref{cr1} below. 

Moreover, the subadditivity of the entropy behaves well when taking the product space, which enables us to  slightly generalize the result to the case of $\H(d,\aa)$ with $d \in \N^*$, as well as products of corank $1$ Carnot groups.

Comparing to the argument in \cite{FP22}, our approach is easier to track the constant. However, we do not know how to deduce the estimates corresponding to the extreme points of the Newton polytope defined in \cite[Section 3]{S11} by our approcah. On the other hand, we don't know how to establish similar Loomis--Whitney inequalities on general Carnot groups using this approach as well.



\subsection{Structure of the paper}\label{ss14}

We recall some basic properties of differential entropy and give the proof of the modified version of the equivalence theorem (namely Theorem \ref{cr1}) between the Brascamp--Lieb inequality and the and subadditivity of the entropy in Section \ref{s2}. In Section \ref{s3} we give the proof of Theorem \ref{t1}. Finally in Section \ref{s4} we give applications to Gagliardo--Nirenberg--Sobolev inequalities and isoperimetric inequalities, as well as generalizations to product spaces.

\section{Preliminaries}\label{s2}
\setcounter{equation}{0}

\subsection{Differential entropy}\label{ss22}

\subsubsection{General definition}

The notion of entropy dates back to mathematical physics \cite{B64} as well as information theory \cite{S48}. For more details about differential entropy, we refer to \cite[Chapter 8]{CT06} or \cite{I93}.

On a measure space $(\Omega, \S, \mu)$, for a non-negative measurable function $f$ on $\Omega$ with $\int_{\Omega} f d\mu = 1$, we define the {\it (differential) entropy} by
\begin{align}\label{defen}
S(f) := \int_\Omega f(x) \ln{f(x)} d\mu(x) = \int_{\supp f} f(x) \ln{f(x)} d\mu(x).
\end{align}
Here $\supp f := \{x \in \Omega: \, f(x) > 0 \}$ and we adopt the convention that $0 \ln{0} = 0$.

Since we will use the modified duality result of \cite{CC09} in the proof of Theorem \ref{t1} so we stick to the notations and definitions there.  In fact, our definition here differs from the original definition of the differential entropy by a negative sign. 

Also notice that the integral in \eqref{defen} does not always exist. For the sake of simplicity, in this paper we always assume that $S(f) \in \R$, or equivalently
\[
\int_{\supp f} f(x) |\ln{f(x)}| d\mu(x) < +\infty.
\]
In this case, we also say that the entropy is finite or the entropy exists. This happens when $f$ is bounded and $\supp f$ is a set of finite measure.

\subsubsection{Entropic inequality}

Assume $\phi$ is a measurable function on $(\Omega,\S,\mu)$ such that $\int_\Omega e^\phi d\mu < +\infty$. Since $s \mapsto \ln{s}$ is strictly concave on $(0,+\infty)$, by Jensen's inequality,
\begin{align*}
\int_\Omega \ln\left( \frac{e^\phi}{f} \right) f d\mu = \int_{\supp f} \ln\left( \frac{e^\phi}{f} \right) f d\mu \le \ln \left( \int_{\supp f} e^\phi d\mu \right) \le \ln \left( \int_\Omega e^\phi d\mu \right).
\end{align*}
Here and in the following we interpret $0 \cdot \infty = 0$. This gives the following proposition (see also \cite[p. 378]{CC09} or \cite[p. 236]{BGL14}).

\begin{proposition}\label{pei}
Assume $S(f) \in \R$ and $\int_\Omega e^\phi d\mu < +\infty$. Then the following inequality holds:
\begin{align}\label{ei}
\int_\Omega f \phi d\mu \le S(f) + \ln \left( \int_\Omega e^\phi d\mu \right),
\end{align}
with the equality attains if and only if $e^\phi = f$ (in the almost everywhere sense).
\end{proposition}

\begin{remark} \label{mi}
In \eqref{ei}, the term $\int_\Omega f \phi d\mu \in [-\infty, +\infty)$ and it could be $-\infty$. A trivial example is that $\phi \equiv -\infty$.
\end{remark}

\subsubsection{Conditional differential entropy}

In this subsection, we recall some basic facts about conditional differential entropy, which will play an important role in our proof of Theorem \ref{t1}. 

If $X$ is a continuous random vector (taking value in $(\Omega,\S,\mu)$) with density $f$ (writing $X \sim f$ in short), we will use $S(X)$ to denote $S(f)$ instead.

Now assume $X$ and $Y$ take values in $(\Omega,\S,\mu)$ and $(\Omega',\S',\mu')$ respectively, and $(X,Y) \sim f$ on $(\Omega \times \Omega', \S \times \S', \mu \times \mu')$. Then we have $Y \sim f_Y$ with the {\it marginal density} 
\begin{align}\label{dmd}
f_Y(y) := \int_\Omega f(x,y) d\mu(x).
\end{align}
For $y \in \supp {f_Y} = \{y \in \Omega': \, f_Y(y) > 0\}$ with $f_Y(y) < +\infty$, the {\it conditional density of $X$ given $Y = y$}  is defined by
\begin{align}\label{dcd}
f(x|y) := \frac{f(x,y)}{f_Y(y)}.
\end{align}
From definition it is easy to see that $\int_\Omega f(\cdot|y) d\mu = 1$.
As a result, for almost every $y \in \supp {f_Y}$, we can define the {\it conditional entropy of $X$ given $Y = y$} by
\begin{align}\label{dce}
S(X|Y = y) := S(f(\cdot|y))
\end{align}
if it exists (that is, $S(f(\cdot|y)) \in \R$). The following result about conditional entropy $S(X|Y = y)$ is not hard to check.

\begin{proposition}[(8.33) of \cite{CT06}] \label{prel}
Suppose $S(X,Y), S(Y) \in \R$ and $(X,Y) \sim f$ on $(\Omega \times \Omega', \S \times \S', \mu \times \mu')$. Then for almost every $y \in \supp f_Y$ we have
\begin{align}\label{relS0}
S(X|Y = y) f_Y(y) + f_Y(y) \ln{f_Y(y)} = \int_\Omega f(x,y) \ln{f(x,y)} d\mu(x).
\end{align}
In particular, we have $S(X|Y = y) \in \R$ for almost every $y \in \supp f_Y$ and 
\begin{align}\label{relS}
S(X,Y) = S(Y) + \int_{\supp f_Y} S(X|Y = y) f_Y(y) d\mu'(y).
\end{align}
\end{proposition}

Assuming also $S(X) \in \R$, since $s \mapsto s \ln{s}$ is strictly convex on $[0,+\infty)$, by Jensen's inequality, we have
\begin{align*}
&\int_{\supp f_Y} S(X|Y = y) f_Y(y) d\mu'(y)  = \int_\Omega \int_{\supp f_Y} f(x|y) \ln{f(x|y)}  f_Y(y) d\mu'(y) d\mu(x) \\
\ge & \, \int_\Omega \left(\int_{\supp f_Y} f(x|y)  f_Y(y) d\mu'(y) \right) \ln  \left(\int_{\supp f_Y} f(x|y)  f_Y(y) d\mu'(y) \right)   d\mu(x)
= S(X).
\end{align*}

This proves the subadditivity of the differential entropy.

\begin{proposition}[Corollary 8.6.2 of \cite{CT06}] \label{psub}
Suppose $S(X,Y), S(X), S(Y) \in \R$. Then we have 
\[
S(X,Y) \ge S(X) + S(Y),
\]
where the equality holds if and only if $X$ and $Y$ are independent.
\end{proposition}

\subsubsection{Pushforward measure and entropy}

Given two measure spaces $(\Omega, \S, \mu)$ and $(M, \M, \nu)$ with a measurable map $p: \Omega \to M$, we consider the pushforward measure $p_\# (f d\mu)$. Now we assume further that $p_\# (f d\mu) \ll \nu$, that is, there exists a density function $f_{(p)}$ such that
\begin{align}\label{rnp}
p_\# (f d\mu) = f_{(p)} d\nu.
\end{align}
Then from the definition of pushforward, for every (bounded) measurable function $\phi$ on $(M, \M, \nu)$, we have
\begin{align}\label{pfe}
\int_\Omega \phi(p(x)) f(x) d\mu(x) = \int_M \phi(z) f_{(p)}(z) d\nu(z).
\end{align}
In particular, choosing $\phi \equiv 1$, it is clear that $\int_M f_{(p)}d\nu = \int_\Omega f d\mu = 1$. Then we can define the  {\it pushforward entropy under $p$} by $S(f_{(p)})$ if it exists (that is, $S(f_{(p)}) \in \R$). 

\begin{remark}\label{exist}
We don't know whether $p_\# (f d\mu) \ll \nu$ in the general case, nor the existence of $S(f_{(p)})$. However, we will prove them for our applications on corank $1$ Carnot groups. See Lemmas \ref{dep} and \ref{dep2} below.
\end{remark}

\begin{remark}\label{rv}
If $X \sim f$ on $(\Omega, \S, \mu)$ and  $p$ is a measurable map from $(\Omega, \S, \mu)$ to $(M, \M, \nu)$ such that \eqref{rnp} holds, we can check directly that $p(X) \sim f_{(p)}$. As a result, using the notation before, we will use $S(p(X))$ to denote the pushforward entropy $S(f_{(p)})$ as well.
\end{remark}

We now prove a consistency result for pushforward and conditional entropy.

\begin{proposition}\label{pcon}
Assume $(X,Y) \sim f$ on $(\Omega \times \Omega', \S \times \S', \mu \times \mu')$ and $p$ is a measurable map from $(\Omega, \S, \mu)$ to $(M, \M, \nu)$. Let the map $\bar{p}$ be a measurable map from $(\Omega \times \Omega', \S \times \S', \mu \times \mu')$ to $(M \times \Omega', \M \times \S', \nu \times \mu')$ defined by $\bar{p}(x,y) = (p(x),y)$. Assume that $\bar{p}_\# (f d\mu d\mu') \ll \nu \times \mu'$ and
\[
\bar{p}_\# (f d\mu d\mu') = f_{(\bar{p})} d\nu d\mu'.
\]
Then for almost every $y \in \supp f_Y$, we have $p_\# (f(\cdot|y) d\mu) \ll \nu$ and 
\[
p_\# (f(\cdot|y) d\mu) = f_{(\bar{p})}(\cdot|y) d\nu.
\]
In other words, for almost every $y \in \supp f_Y$, we have $f(\cdot|y)_{(p)} = f_{(\bar{p})}(\cdot|y)$.
\end{proposition}

\begin{proof}
It follows from our assumption that for every measurable functions $\phi$ and $\psi$ on $(M, \M, \nu)$ and $(\Omega', \S', \mu')$ respectively, we have
\[
\int_\Omega \int_{\Omega'} f(x,y) \phi(p(x)) \psi(y) d\mu(x) d\mu'(y) = \int_M  \int_{\Omega'}  f_{(\bar{p})}(z,y) \phi(z) \psi(y) d\nu(z) d\mu'(y).
\]
Since $\psi$ is arbitrary, it follows that for almost every $y$,
\begin{align}\label{inter}
\int_\Omega  f(x,y) \phi(p(x))  d\mu(x) = \int_M   f_{(\bar{p})}(z,y) \phi(z)  d\nu(z).
\end{align}
Choosing $\phi \equiv 1$, we obtain that $f_Y(y) = (f_{(\bar{p})})_Y(y) < +\infty$ holds for almost every $y$. Then dividing both side of \eqref{inter} by $f_Y(y)$ for such $y \in \supp f_Y$, it deduces that for almost every $y \in \supp f_Y$
\[
\int_\Omega  f(x|y) \phi(p(x))  d\mu(x) = \int_M   \frac{f_{(\bar{p})}(z,y)}{(f_{(\bar{p})})_Y(y)} \phi(z)  d\nu(z) = \int_M   f_{(\bar{p})}(z|y) \phi(z)  d\nu(z),
\]
which proves the proposition.
\end{proof}

\begin{remark}
In fact, under some mild conditions, Proposition \ref{pcon} allows us to write $S(p(X)|Y = y)$ without ambiguities. More precisely, it can be interpreted either as the pushforward entropy of $X$ given $Y = y$ under $p$, or as the conditional entropy of $p(X)$ given $Y = y$.
\end{remark}






\subsection{The duality result in \cite{CC09}}\label{ss23}

The main tool of this article is \cite[Theorem 2.1]{CC09}, which states the duality of the Brascamp--Lieb inequality and the subadditivity of the entropy. Although it is not written explicitly in \cite{CC09}, we should be careful about the case where the entropies are not finite since otherwise some unexpected operation such as $-\infty + \infty$ will appear in the proof.

As a result, for the sake of rigor and completeness, and also for the modification on corank $1$ Carnot group (see Theorem \ref{cr1} below), we state \cite[Theorem 2.1]{CC09} again and include a complete proof here. 

\begin{theorem}[Theorem 2.1 of \cite{CC09}]\label{tcc09}
Let $(\Omega, \S, \mu)$ be a measure space, $m \ge 1$, and for $1 \le j \le m$, $(M_j,\M_j,\nu_j)$ be a measure space together with a measurable map $p_j$ from $\Omega$ to $M_j$. Fix $D \in \R$ and $c_j > 0$, $1 \le j \le m$. 
\begin{enumerate}[(i)]
\item If for any $m$ non-negative measurable functions $f_j : M_j \to [0, +\infty)$, $1 \le j \le m$, we have
\begin{align}\label{gbl}
\int_\Omega \prod_{j = 1}^m f_j(p_j(x)) d\mu(x) \le e^D \prod_{j = 1}^m \left( \int_{M_j} f_j^{1/c_j}(t) d\nu_j(t)\right)^{c_j},
\end{align}
then the following subadditivity of the entropy holds
\begin{align}\label{sae}
\sum_{j = 1}^m c_j S(f_{(p_j)}) \le S(f) + D
\end{align}
for all probability density $f$ belonging to the following set
\begin{align}\label{defw}
\W := \{f : \,  S(f) \in \R, \mbox{ and } \forall \, 1 \le j \le m, (p_j)_\#(f d\mu) =  f_{(p_j)} d\nu_j \mbox{ with } S(f_{(p_j)}) \in \R \}.
\end{align}

\item Conversely, if \eqref{sae} holds for all probability density $f \in \W_0 \subset \W$, then \eqref{gbl} holds for all $m$ non-negative functions $f_j : M_j \to [0, +\infty)$, $1 \le j \le m$ satisfying that 
\begin{align}\label{assfj}
\int_{M_j} f_j^{1/c_j} d\nu_j < +\infty, \quad 0 < \int_\Omega \prod_{j = 1}^m f_j \circ p_j  d\mu  < +\infty, \quad \mbox{and} \quad \frac{\prod_{j = 1}^m f_j \circ p_j}{\int_\Omega \prod_{j = 1}^m f_j \circ p_j  d\mu} \in \W_0.
\end{align}
\end{enumerate}
\end{theorem}

\begin{proof}
We first prove the first assertion. In fact, for $f \in \W$, the function $f_{(p_j)}$ is a non-negative function on $M_j$ with $\int_{M_j} f_{(p_j)} d\nu_j  =1$. As a result, it we choose $f_j = f_{(p_j)}^{c_j}$, then \eqref{gbl} gives
\[
\int_\Omega \prod_{j = 1}^m f_{(p_j)}^{c_j}(p_j(x)) d\mu(x) \le e^D < +\infty.
\]
It follows from the inequality above and Proposition \ref{pei} that 
\begin{align*}
&D + S(f) \ge \ln \left( \int_\Omega \prod_{j = 1}^m f_{(p_j)}^{c_j}(p_j(x)) d\mu(x) \right)  + S(f) 
\ge   \int_\Omega f(x) \ln \left( \prod_{j = 1}^m f_{(p_j)}^{c_j}(p_j(x))  \right) d\mu(x) \\
= & \, \sum_{j = 1}^m c_j  \int_\Omega f(x) \ln f_{(p_j)}(p_j(x)) d\mu(x) 
=  \sum_{j = 1}^m c_j  \int_{M_j} f_{(p_j)}(t) \ln f_{(p_j)}(t) d\nu_j(t) = \sum_{j = 1}^m c_j  S(f_{(p_j)}),
\end{align*}
where we have used \eqref{pfe} in the penultimate ``$=$''. Then we turn to prove the second assertion. In fact, for $f \in \W_0$ and $f_j, 1 \le j \le m$ satisfying \eqref{assfj}, by Proposition \ref{pei} and \eqref{pfe} again we have 
\[
c_j \ln \left( \int_{M_j} f_j^{1/c_j}(t) d\nu_j(t) \right) + c_j S(f_{(p_j)}) \ge \int_{M_j} f_{(p_j)}(t) \ln f_j(t) d\nu_j(t) = \int_\Omega f(x) \ln f_j(p_j(x))  d\mu(x).
\]
Adding $j$ from $1$ to $m$ (which is possible by Remark \ref{mi}), we obtain
\[
\ln \prod_{j = 1}^m \left( \int_{M_j} f_j^{1/c_j}(t) d\nu_j(t)\right)^{c_j} + \sum_{j = 1}^m c_j S(f_{(p_j)}) \ge  \int_\Omega f(x) \ln \prod_{j = 1}^m f_j(p_j(x))  d\mu(x)
\]
Then writing $F = \prod_{j = 1}^m f_j \circ p_j$, combining the inequality above with \eqref{sae} we obtain
\begin{align}\label{mcc1}
D + \ln \prod_{j = 1}^m \left( \int_{M_j} f_j^{1/c_j}(t) d\nu_j(t)\right)^{c_j} + S(f) \ge \int_\Omega f(x) \ln F(x)  d\mu(x).
\end{align}
From assumption we can choose $f = F/ \int_\Omega F d\mu \in \W_0$. By a direct computation, we have
\begin{align}\label{mcc2}
\int_\Omega f(x) \ln F(x)  d\mu(x) - S(f) = \ln \left( \int_\Omega F(x) d\mu(x)\right).
\end{align}
Combining \eqref{mcc1} with \eqref{mcc2}, we prove \eqref{gbl} under \eqref{sae} and \eqref{assfj}.
\end{proof}

\begin{remark}
Comparing to \cite[Theorem 2.1]{CC09}, we added the extra set $\W_0$ for the sake of the modification for corank $1$ Carnot groups in Theorem \ref{cr1} below. In fact, for our application, we do not need to establish \eqref{sae} for the whole $f \in \W$ but only on a suitable subset $\W_0$. 
\end{remark}

\begin{remark}\label{forre}
Without further assumptions on the space, it is hard for us to use an approximation process to remove the restriction \eqref{assfj} since we do not know whether the set $\W$ is large enough.
\end{remark}

Fortunenately, on Euclidean spaces $(\R^k, \B_k, \m_k)$, the following two lemmas show that the set
\begin{align}\label{defDk}
 \D_k := \{ f: \, \mbox{$f \ge 0$, $f$, as well as $\supp f$, is bounded} \}
\end{align}
behaves well under Euclidean projections as well as nonlinear projections defined in \eqref{pro2} and \eqref{pro3}. Here $k \in \N^*$, $\B_k$ and $\m_k$ denote the corresponding Borel $\sigma$-algebra and Lebesgue measure on $\R^k$. 

\begin{lemma}\label{dep}
Let $k,k' \in \N^*$ and $k' < k$. Assume $f \in \D_k$ and $\P$ is a projection from $\R^k$ onto $\R^{k'}$ by deleting some coordinates of $\R^k$. Then we have $\P_\# (f d\m_k) \ll \m_{k'}$ and $f_{(\P)} \in \D_{k'}$, where $\P_\# (f d\m_k) = f_{(\P)} d\m_{k'}$.
\end{lemma}

\begin{proof}
Without loss of generality we can assume that $\P$ is given by deleting the last $k - k'$ coordinates. In the following we write an element $x$ in $\R^k$ as $(x',x'')$ with $x' \in \R^{k'}$ and $x'' \in \R^{k - k'}$. By definition we have $\P(x) = x'$. Then for every measurable function $\phi$ on $\R^{k'}$ we have
\[
\int_{\R^k} f(x) \phi(\P(x)) dx = \int_{\R^{k'}}\int_{\R^{k - k'}} f(x) \phi(x') dx'dx''
=  \int_{\R^{k'}} \left(\int_{\R^{k - k'}} f(x',x'') dx''  \right)  \phi(x') dx',
\]
which implies $\P_\# (f d\m_k) \ll \m_{k'}$ and 
\[
f_{(\P)}(x') = \int_{\R^{k - k'}} f(x',x'') dx''.
\]
It is clear that $f_{(\P)} \in \D_{k'}$ if $f \in \D_k$.
\end{proof}

\begin{lemma}\label{dep2}
On corank $1$ Carnot group $\H(d,\aa)$, assume $f \in \D_{d + 2n + 1}$ and $d + 1 \le j \le d + 2n$. Then we have $(\pi_j)_\# (f d\m_{d + 2n + 1}) \ll \m_{d + 2n}$ and $f_{(\pi_j)} \in \D_{d + 2n}$, where $(\pi_j)_\# (f d\m_{d + 2n + 1}) = f_{(\pi_j)}
 d\m_{d + 2n}$.
\end{lemma}

\begin{proof}
For the ease of the notation we only prove the special case $d = 0$ and $ j = 1$ and the proof for the case $d \in \N^*$ and the other projections are similar. In the following we write an element $x$ in $\R^{2n}$ as $(x_1,\h{x}_1)$ with $x_1 \in \R$ and $\h{x}_1 \in \R^{2n - 1}$. As a result, $\pi_1(x,t) = (\h{x}_1, t + \frac{\a_1}{2} x_1 x_2)$. Then for every measurable function $\phi$ on $\R^{2n}$ we have
\begin{align*}
&\int_{\R^{2n}} \int_{\R} f(x,t) \phi(\pi_1(x,t)) dxdt \\
=& \, \int_\R  \int_{\R^{2n - 1}} \int_\R f(x,t)\phi\left(\h{x}_1, t + \frac{\a_1}{2} x_1 x_2\right) dx_1d\h{x}_1dt \\
=& \,  \int_\R  \int_{\R^{2n - 1}} \int_\R f\left(x, t - \frac{\a_1}{2} x_1 x_2\right) \phi(\h{x}_1, t) dx_1d\h{x}_1dt \\
=& \, \int_\R  \int_{\R^{2n - 1}} \left[ \int_\R f\left(x_1, \h{x}_1, t - \frac{\a_1}{2} x_1 x_2\right)  dx_1 \right] \phi(\h{x}_1, t)d\h{x}_1dt.
\end{align*}
It follows that $(\pi_1)_\# (f d\m_{2n + 1}) \ll \m_{2n}$ and 
\begin{align}\label{expfpi}
f_{(\pi_1)}(\h{x}_1,t) = \int_\R f\left(x_1, \h{x}_1, t - \frac{\a_1}{2} x_1 x_2\right)  dx_1.
\end{align}
Now we are in a position to prove $f_{(\pi_1)} \in \D_{2n}$. Since $f \in \D_{2n + 1}$, we can assume that $|f| \le M$ and $\supp f \subset [-L,L]^{2n + 1}$ for some $M, L > 0$. Then the integral in \eqref{expfpi} is actually on $[-L,L]$ and we have $|f_{(\pi_1)} | \le 2LM$, which implies $f_{(\pi_1)}$ is bounded. Furthermore, if $|x_j| > L$ for some $2 \le j \le 2n$, then the integrand in \eqref{expfpi} is $0$ and thus $f_{(\pi_1)}$ is also $0$. This proves $\supp f_{(\pi_1)} \subset [-L,L]^{2n - 1} \times \R$. If $|t| > L + \frac{\a_1}{2} L^2$ and $|x_2| \le L$, then for any $x_1 \in [-L,L]$, 
\[
\left| t - \frac{\a_1}{2} x_1 x_2 \right| > L + \frac{\a_1}{2} L^2 - \frac{\a_1}{2} L^2 = L,
\]
which implies the integrand in \eqref{expfpi} is $0$ again. As a result, we proved $\supp f_{(\pi_1)} \subset [-L,L]^{2n - 1} \times \left[- L - \frac{\a_1}{2} L^2,L + \frac{\a_1}{2} L^2\right]$. In conclusion, we have $f_{(\pi_1)} \in \D_{2n}$.
\end{proof}

Now we focus on our corank $1$ Carnot group $\H(d,\aa)$ case. Combining Lemmas \ref{dep} and \ref{dep2} we obtain the following corollary.

\begin{corollary}\label{cW}
On corank $1$ Carnot group $\H(d,\aa)$ with projections $\{\pi_j\}_{j = 1}^{d + 2n}$, we have 
\begin{align}\label{defDP}
\DP_{d + 2n + 1} := \{f \in \D_{d + 2n + 1}: \, \mbox{$f$ is a probability density}\} \subset \W.
\end{align}
\end{corollary}

Now we state Theorem \ref{tcc09} on $\H(d,\aa)$ in a more concise way, that is, without the annoying restriction \eqref{assfj}. 

\begin{theorem} \label{cr1}
On corank $1$ Carnot group $\H(d,\aa)$ with projections $\{\pi_j\}_{j = 1}^{d + 2n}$, fix $D \in \R$ and $c_j > 0$, $1 \le j \le d + 2n$. Then the following two assertions are equivalent: 
\begin{enumerate}[(1)]
\item For any $d + 2n$ non-negative measurable functions $f_j : \R^{d + 2n} \to [0, +\infty)$, $1 \le j \le d + 2n$, we have
\begin{align}\label{gbl2}
\int_{\H(d,\aa)} \prod_{j = 1}^{d + 2n} f_j ( \pi_j(x,t)) dxdt \le e^D \prod_{j = 1}^{d + 2n} \left( \int_{\R^{d + 2n}} f_j^{1/c_j}(\h{x}_j,t) d\h{x}_j dt \right)^{c_j}.
\end{align}

\item For every $f \in \DP_{d + 2n + 1}$, the following subadditivity of the entropy holds
\begin{align}\label{sae2}
\sum_{j = 1}^m c_j S(f_{(\pi_j)}) \le S(f) + D.
\end{align}

\end{enumerate}
\end{theorem}

\begin{proof}
(1) $\Rightarrow$ (2): It just follows from (i) of Theorem \ref{tcc09} and Corollary \ref{cW}. \\
(2) $\Rightarrow$ (1): We claim that if $f_j \in \D_{d + 2n}$ for $1 \le j \le d + 2n$, then $F = \prod_{j = 1}^{d + 2n} f_j \circ \pi_j \in \D_{d + 2n + 1}$.  In fact, it is easy to see that $F$ is a bounded function. Now assume that $L > 0$ is a large number such that $\cup_{j = 1}^{d + 2n} \supp f_j \subset [-L, L]^{d + 2n}$. If $(x,t) \in \supp F$, then for every $1 \le j \le d + 2n$, $f_j(\pi_j(x,t)) > 0$. This implies $|x_j| \le L$ for every $1 \le j \le d + 2n$. Furthermore, since $n \ge 1$ we must have $f_{d + 1}(\pi_{d + 1}(x,t)) > 0$ and $f_{d + 2}(\pi_{d + 2}(x,t)) > 0$. Thus 
we obtain
\[
\left| t + \frac{\a_1}{2} x_{d + 1} x_{d + 2}\right| \le L , \qquad 
\left| t - \frac{\a_1}{2} x_{d + 1} x_{d + 2}\right| \le L,
\]
which implies $|t| \le L$ since
\[
2t^2 \le \left| t + \frac{\a_1}{2} x_{d + 1} x_{d + 2}\right|^2 + \left| t - \frac{\a_1}{2} x_{d + 1} x_{d + 2}\right|^2 \le 2L^2.
\]
This proves the claim. As a result,  if $\int F > 0$, then we have $F/\int F \in \DP_{d + 2n + 1}$. Thus, \eqref{assfj} holds for $\W_0 = \DP_{d + 2n + 1}$. In this case  \eqref{gbl2} holds by (ii) of Theorem \ref{tcc09}. In the opposite case $\int F = 0$,  \eqref{gbl2} holds automatically.

In conclusion, we proved that when  $f_j \in \D_{d + 2n}$ for $1 \le j \le d + 2n$, \eqref{gbl2} holds. The general case can be obtained by first applying \eqref{gbl2} to the truncated functions $f_j \chi_{\{x: \, |x| \le k, f_j(x) \le k\}}, 1 \le j \le d + 2n$ and then letting $k \to +\infty$.

\end{proof}

\begin{remark}\label{finite}
Using Theorem \ref{cr1}, we reduce the proof of the main theorem to the subadditivity of the entropy on a relatively simple set. Due to Lemmas \ref{dep} and \ref{dep2}, as well as Proposition \ref{prel}, all the entropies appearing in the proof below are finite (or at least finite almost everywhere when there is an integral).
\end{remark}

\begin{remark}\label{prod}
Actually with similar arguments we can prove that Theorem \ref{cr1} is also valid on products of corank $1$ Carnot groups.
\end{remark}

Now we are prepared to prove our main theorem.

\section{Proof of Theorem \ref{t1}}\label{s3}
\setcounter{equation}{0}

\subsection{Proof for the special case $d = 0$}\label{ss31}

We first consider the special case $d = 0$. General case $d \in \N^*$ will be treated in next subsection where we will see how the subadditivity of the entropy behaves under taking product spaces.

From Theorem \ref{cr1}, it suffices to show that for every $f \in \DP_{2n + 1}$, the following subadditivity of the entropy holds
\begin{align}\label{target}
 \sum_{j =  1}^{ 2n}  \frac{n + 1}{n( 2n + 1)} S(f_{(\pi_j)}) \le S(f) + \ln \CC(0,\aa).
\end{align}

To begin with, by a change of variable $t \mapsto \alpha t$ and applying Theorem \ref{tH1} to functions $f_1(\cdot, \alpha \, \cdot)$ , $f_2(\cdot, \alpha \, \cdot)$,  we obtain that for every $\a > 0$, 
\begin{align}\label{LWH1p}
\int_{\R^3} f_1\left(y, t + \frac{\a}{2} xy\right) f_2\left(x, t - \frac{\a}{2}xy\right)  dxdydt \le  \frac{\|\RR\|_{\frac{3}{2} \to 3}}{\a^{\frac{1}{3}}} \|f_1\|_{\frac{3}{2}} \|f_2\|_{\frac{3}{2}},
\end{align}
which by (i) of Theorem \ref{tcc09} implies
\begin{align}\label{sea3}
\frac{2}{3} S\left(Y_2, W + \frac{\a}{2} Y_1 Y_2 \right) + \frac{2}{3} S\left(Y_1, W - \frac{\a}{2} Y_1 Y_2 \right) \le S(Y_1, Y_2, W) + \CC_0 - \frac{1}{3} \ln{\a}
\end{align}
holds for all the entropies appearing above finite, where $\CC_0 := \ln\left( \|\RR\|_{\frac{3}{2} \to 3} \right)$.

In the following for $1 \le i < j \le  2n$, we use $\P_{i,j}$ to denote the projection by deleting the $i$-th and $j$-th coordinates. Furthermore, for the ease of the notation we assume $(X,T) \sim f \in \DP_{2n + 1}$ with $X = (X_1, \ldots, X_{2n})$. From Remark \ref{finite} we know that all the entropies appearing in the proof are finite so we do not need to care about the finiteness problem.

Now for every $1 \le j \le n$, by Proposition \ref{prel} we have
\begin{align*}
S(\pi_{2j - 1}(X,T)) = S(X^{j}) + \int_{\supp f_{X^{j}}} S\left(X_{2j}, T + \frac{\a_j}{2} X_{2j - 1}X_{2j} \Big| X^{j} = y\right)  f_{X^{j}}(y) dy, \\
S(\pi_{2j }(X,T)) = S(X^{j}) + \int_{\supp f_{X^{j}}} S\left(X_{2j - 1}, T - \frac{\a_j}{2} X_{2j - 1}X_{2j} \Big| X^{j} = y\right)  f_{X^{j}}(y) dy,
\end{align*}
where $X^{j} := \P_{2j - 1, 2j}(X)$. Adding the two equations above and noticing that for almost every $y \in \supp f_{X^{j}}$ we deduce from \eqref{sea3} as well as Proposition \ref{pcon} that 
\begin{align*}
&S\left(X_{2j}, T + \frac{\a_j}{2} X_{2j - 1}X_{2j} \Big| X^{j} = y\right)  +  S\left(X_{2j - 1}, T - \frac{\a_j}{2} X_{2j - 1}X_{2j} \Big| X^{j} = y\right)  \\
\le & \, \frac{3}{2} \left( S\left(X_{2j - 1}, X_{2j}, T  \Big| X^{j} = y\right) + \CC_0 \right) - \frac{1}{2} \ln{\a_j}.
\end{align*}
Thus we can obtain the following
\begin{align}\nonumber
&S(\pi_{2j - 1}(X,T)) + S(\pi_{2j }(X,T))  \\
\nonumber
\le & \,2 S(X^{j}) + \frac{3}{2} \int_{\supp f_{X^{j}}} S\left(X_{2j - 1}, X_{2j}, T  \Big| X^{j} = y\right) f_{X^{j}}(y) dy + \frac{3}{2} \CC_0 - \frac{1}{2} \ln{\a_j}\\
\label{main1}
= & \frac{1}{2 }S(X^{j}) + \frac{3}{2} S(X,T) +  \frac{3}{2} \CC_0 - \frac{1}{2} \ln{\a_j},
\end{align}
where in the last ``$=$'' we have used Proposition \ref{prel} again. Adding \eqref{main1} from $j  = 1$ to $n$ yields
\begin{align}\label{main2}
\sum_{j = 1}^{2n} S(\pi_{j}(X,T)) \le \frac{1}{2} \sum_{j = 1}^n S(\P_{2j - 1, 2j}(X)) + \frac{3n}{2} S(X,T) + \frac{3n}{2} \CC_0 - \sum_{j = 1}^n\frac{1}{2} \ln{\a_j}.
\end{align}

Noticing that 
\[
\P_{2j - 1, 2j}\P_{2j - 1, 2j}^* =  \id_{2n - 2}, \qquad  \frac{1}{n - 1}\sum_{j = 1}^n \P_{2j - 1, 2j}^* \P_{2j - 1, 2j} = \id_{2n},
\] 
it follows from Theorem \ref{tfBL} (the part for the geometric Brascamp--Lieb inequality), together with (i) of Theorem \ref{tcc09} that 
\begin{align}\label{main3}
 \sum_{j = 1}^n S(\P_{2j - 1, 2j}(X)) \le ( n - 1) S(X).
\end{align}

Finally, from Proposition \ref{pWL} (with $d = 0$), together with (i) of Theorem \ref{tcc09} again, we obtain
\begin{align}\label{main4}
\sum_{j = 1}^{2n + 1} S(\pi_j(X,T)) \le 2n  S(X,T).
\end{align}

Recalling that $\pi_{2n + 1}(X,T) = X$ by definition (see \eqref{pro1}), combining \eqref{main2}-\eqref{main4}, we obtain that 
\begin{align*}
&\sum_{j = 1}^{2n} S(\pi_{j}(X,T)) \le \frac{1}{2} \sum_{j = 1}^n S(\P_{2j - 1, 2j}(X)) + \frac{3n}{2} S(X,T) + \frac{3n}{2} \CC_0 - \sum_{j = 1}^n\frac{1}{2} \ln{\a_j} \\
\le & \, \frac{n - 1}{2}  S(X) + \frac{3n}{2} S(X,T) + \frac{3n}{2} \CC_0 - \sum_{j = 1}^n\frac{1}{2} \ln{\a_j} \\
\le & \, \frac{n - 1}{2} \left[ 2n S(X,T) - \sum_{j = 1}^{2n} S(\pi_{j}(X,T)) \right]  + \frac{3n}{2} S(X,T) + \frac{3n}{2} \CC_0 - \sum_{j = 1}^n\frac{1}{2} \ln{\a_j},
\end{align*}
which implies \eqref{target} with the constant $\ln \CC(0,\aa) = \frac{3}{2n + 1} \CC_0 - \frac{\sum_{j = 1}^n \ln{\a_j}}{n(2n + 1)}$.

\subsection{Proof for general case $d \in \N^*$}

If we could establish the following two propositions, then our Theorem \ref{t1} follows by applying Proposition \ref{ppro} to $\Omega = \R^d$ (with \eqref{LWf}) and $\Omega' = \H(0,\aa)$ for $d \ge 2$
and Proposition \ref{ppro2} to $\Omega = \H(0,\aa)$ for $d = 1$
since $\H(d,\aa) \cong \R^d \times \H(0,\aa)$.

\begin{proposition}\label{ppro}
Assume $(\Omega, \S, \mu)$ is either a Euclidean space whose dimension is greater than $2$ or a corank $1$ Carnot group and $\{p_j\}_{j = 1}^m$ are corresponding $\{\P_j\}_{j = 1}^m$ or $\{\pi_j\}_{j = 1}^m$. The same assumption goes to $(\Omega', \S', \mu')$ with $\{p_\ell'\}_{\ell = 1}^{m'}$. Furthermore, suppose there exist $c_j > 0 \, ( 1 \le j \le m)$ and $c_\ell' > 0 \, ( 1 \le \ell \le m')$ and $D, D' \in \R$ such that 
\begin{align}\label{gbl4}
\int_\Omega \prod_{j = 1}^m f_j(p_j(x)) d\mu(x) \le e^D \prod_{j = 1}^m \|f_j\|_{1/c_j}, \\
\label{gbl3}
\int_{\Omega'} \prod_{\ell = 1}^{m'} g_\ell(p_\ell'(y)) d\mu'(y) \le e^{D'} \prod_{\ell = 1}^{m'} \|g_\ell\|_{1/{c_\ell'}},
\end{align} 
hold for all non-negative measurable functions $f_1, \ldots, f_m, g_1, \ldots, g_{m'}$. 
Then on the product space  $(\Omega \times \Omega', \S \times \S', \mu \times \mu')$ we have 
\[
\int_{\Omega \times \Omega'} \prod_{j = 1}^{m + m'} f_j(\bar{p}_j(x,y)) d\mu(x)d\mu'(y) \le e^{\bar{D}}  \prod_{j = 1}^{m + m'} \|f_j\|_{1/{\bar{c}_j}}
\]
for all non-negative measurable functions $f_1, \ldots, f_{m + m'}$, where \[\bar{D} := \frac{ (\sum_{\ell = 1}^{m'} c_\ell' - 1)D + (\sum_{j = 1}^m c_j - 1)D'}{\left(\sum_{j = 1}^m c_j\right)\left(\sum_{\ell = 1}^{m'} c_\ell'\right) - 1},\]
\begin{align}\label{defbp}
\bar{p}_j(x,y)  := \begin{cases}
(p_j(x),y) & \mbox{if $1 \le j \le m$} \\
(x,p_{j -m}'(y)) & \mbox{if $m+1 \le j \le m + m'$}
\end{cases},
\end{align}
and
\[ 
\bar{c}_j := \begin{cases}
\frac{c_j (\sum_{\ell = 1}^{m'} c_\ell' - 1)}{\left(\sum_{j = 1}^m c_j\right)\left(\sum_{\ell = 1}^{m'} c_\ell'\right) - 1} & \mbox{if $1 \le j \le m$} \\
\frac{c_{j - m}' (\sum_{j = 1}^m c_j - 1)}{\left(\sum_{j = 1}^m c_j\right)\left(\sum_{\ell = 1}^{m'} c_\ell'\right) - 1} & \mbox{if $m+1 \le j \le m + m'$}
\end{cases}.
\] 
\end{proposition}

\begin{proof}
As before, by Theorem \ref{cr1} and Remark \ref{prod}, we only need to prove
\begin{align}\label{targett}
\sum_{j = 1}^{m + m'} \bar{c}_j S(f_{\bar{p}_j}) \le S(f) + \bar{D}, \qquad \forall \, f \in \DP_k,
\end{align}
where $k$ is the topological dimension of $\Omega \times \Omega'$.
Assume $(X,Y) \sim f \in \DP_k$. Then similar to the argument in Subsection \ref{ss31}, by Propositions \ref{prel} and \ref{pcon}, and \eqref{gbl4}, together with (i) of Theorem \ref{tcc09} and Remark \ref{finite}, we obtain  
\begin{align*}
&\sum_{j = 1}^m c_j S(\bar{p}_j(X,Y)) = \sum_{j = 1}^m c_j S(p_j(X),Y) \\
=& \, \sum_{j = 1}^m c_j \left( S(Y) + \int_{\supp f_Y} S(p_j(X)|Y=y) f_Y(y) dy  \right) \\
=& \,  \left(\sum_{j = 1}^m c_j \right) S(Y) + \int_{\supp f_Y} \left( \sum_{j = 1}^m c_j S(p_j(X)|Y=y) \right)f_Y(y) dy \\
\le & \, \left(\sum_{j = 1}^m c_j \right) S(Y) + \int_{\supp f_Y} S(X|Y = y)  f_Y(y) dy + D \\
= & \, \left(\sum_{j = 1}^m c_j - 1 \right) S(Y) + S(X,Y) + D.
\end{align*}
Similarly, we obtain
\[
\sum_{\ell =  1}^{ m'} c_\ell' S(\bar{p}_{m + \ell}(X,Y)) \le \left(\sum_{\ell = 1}^{m'} c_\ell' - 1 \right) S(X) + S(X,Y) + D'.
\]
Writing $A = \sum_{j = 1}^m c_j - 1 $ and $B = \sum_{\ell = 1}^{m'} c_\ell' - 1$ ($A,B > 0$ by dilation invariance, see Remark \ref{coeffr} below), with the two inequalities above we obtain
\begin{align*}
&\sum_{j = 1}^m c_j B S(\bar{p}_j(X,Y)) + \sum_{\ell =  1}^{ m'} c_\ell' A S(\bar{p}_{m + \ell}(X,Y)) \\
\le & \, AB ( S(X) + S(Y))  + (A + B) S(X,Y) + BD + AD' \\ \le & \, (A + B + AB) S(X,Y) + (BD + AD'),
\end{align*}
where we have used the subadditivity of the differential entropy (Proposition \ref{psub}) in the last ``$\le $''. This gives desired \eqref{targett} and thus proves this proposition.
\end{proof}

Similar and simpler argument gives the following proposition as well.

\begin{proposition}\label{ppro2}
Assume $(\Omega, \S, \mu)$ is a corank $1$ Carnot group and $\{p_j\}_{j = 1}^m$ are the corresponding projections.  Furthermore, suppose there exist $c_j > 0 \, ( 1 \le j \le m)$ and $D \in \R$ such that 
\begin{align}\label{gbl42}
\int_\Omega \prod_{j = 1}^m f_j(p_j(y)) d\mu(y) \le e^D \prod_{j = 1}^m \|f_j\|_{1/c_j}, 
\end{align} 
holds for all non-negative measurable functions $f_1, \ldots, f_m$. 
Then on the product space  $(\R \times \Omega, \B_1 \times \S, \m_1 \times \mu)$ we have 
\[
\int_{\R \times \Omega} \prod_{j = 1}^{m + 1} f_j(\bar{p}_j(x,y)) d\m_1(x)d\mu(y) \le e^{\bar{D}}  \prod_{j = 1}^{m + 1} \|f_j\|_{1/{\bar{c}_j}}
\]
for all non-negative measurable functions $f_1, \ldots, f_{m + 1}$, where $\bar{D} := \frac{D}{\sum_{j = 1}^m c_j }$,
\begin{align}\label{defbp2}
\bar{p}_j(x,y)  := \begin{cases}
y & \mbox{if $j  = 1$} \\
(x,p_{j - 1}(y)) & \mbox{if $2 \le j \le m + 1$}
\end{cases},
\end{align}
and
\[ 
\bar{c}_j := \begin{cases}
\frac{\sum_{j = 1}^m c_j - 1}{\sum_{j = 1}^m c_j} & \mbox{if $j = 1$} \\
\frac{c_{j - 1} }{\sum_{j = 1}^m c_j} & \mbox{if $2 \le j \le m + 1$}
\end{cases}.
\] 

\end{proposition}

\begin{remark}\label{coeffr}
In fact, for Euclidean spaces or corank $1$ Carnot groups (without the last projection), by dilation invariance, we automatically have 
\[
\sum_{j = 1}^m c_j = \frac{Q}{Q - 1}, \qquad \sum_{\ell = 1}^{m'} c_\ell' = \frac{Q'}{Q' - 1},
\]
where $Q$ and $Q'$ are homogeneous dimensions of $(\Omega, \S, \mu)$ and $(\Omega', \S', \mu')$ respectively. This simplifies the constants in Proposition \ref{ppro} by
\[ 
\bar{D} := \frac{ D(Q - 1) + D'(Q'-1)}{Q + Q' - 1},
\qquad 
\bar{c}_j := \begin{cases}
\frac{c_j (Q - 1)}{Q + Q' - 1} & \mbox{if $1 \le j \le m$} \\
\frac{c_{j - m}' (Q' - 1)}{Q + Q' - 1} & \mbox{if $m+1 \le j \le m + m'$}
\end{cases}
\] 
and the constants in Proposition \ref{ppro2} by
\[
\bar{D} := \frac{ D(Q - 1) }{Q },
\qquad 
\bar{c}_j := \begin{cases}
\frac{1}{Q } & \mbox{if $j = 1$} \\
\frac{c_{j - 1} (Q - 1)}{Q } & \mbox{if $2 \le j \le m + 1$}
\end{cases}.
\]
However, for the sake of possible extensions to general cases, we decided to state Propositions \ref{ppro} and \ref{ppro2} in the current way.
\end{remark}

\section{Applications and generalizations}\label{s4}
\setcounter{equation}{0}

\subsection{Applications to Gagliardo--Nirenberg--Sobolev inequalities and isoperimetric inequalities}

In this subsection we use our Loomis--Whitney inequality (cf. \eqref{LWco1}) to deduce the Gagliardo--Nirenberg--Sobolev inequality as well as the isoperimetric inequality on corank $1$ Carnot groups. In fact, it follows from a quite standard argument and we just give a brief proof here. For more details for the proof as well as the history, we refer to \cite{CDPT07, FP22, S02} and the references therein.  

We begin by stating a direct corollary of Theorem \ref{t1}, which is the Loomis--Whitney inequality for the sets. See also Remark \ref{rset} for the case of Euclidean spaces.
\begin{corollary}\label{c1}
On corank $1$ Carnot group $\H(d,\aa)$, it holds that
\begin{align}\label{LWco12}
\m_{d + 2n + 1}(E) \le \CC(d,\aa) \prod_{j = 1}^d \m_{d + 2n}(\pi_j(E))^{\frac{1}{d + 2n + 1}} \prod_{j = d + 1}^{d + 2n} \m_{d + 2n}(\pi_j(E))^{\frac{n + 1}{n(d + 2n + 1)}},
\end{align}
for all measurable set $E$, with the constant $\CC(d,\aa)$ defined in \eqref{defCC}.
\end{corollary}

\begin{proof}
Noticing that for every measurable set $E$ we have
\[
E \subset \cap_{j = 1}^{d + 2n} \pi_j^{-1}(\pi_j(E)), \quad \mbox{which implies} \quad 
\chi_E \le \prod_{j = 1}^{d + 2n}  \chi_{\pi_j(E)} \circ \pi_j,
\]
we prove the corollary by  choosing $f_j = \chi_{\pi_j(E)}$ in \eqref{LWco1}. 
\end{proof}

To proceed we need definitions on functions with bounded variation and the perimeter for sets. We use $\F(\H(d,\aa))$ to denote the set of functions $\varphi \in C^1_0(\H(d,\aa), \R^{d + 2n})$ such that $|\varphi| \le 1$, and define {\it variation of a function $f \in L^1(\H(d,\aa))$} by 
\[
\Var_{\H(d,\aa)} (f) := \sup_{\varphi \in \F(\H(d,\aa))} \int_{\H(d,\aa)} f(x,t) \sum_{j = 1}^{d + 2n} \X_j \varphi_j(x,t) dxdt.
\]
Now we use $\BV(\H(d,\aa))$ to denote the space of all functions $f \in L^1(\H(d,\aa))$ with finite variation. It is a Banach space with the following natural norm:
\[
\|f\|_{\BV(\H(d,\aa))} := \|f\|_1 + \Var_{\H(d,\aa)} (f).
\]
Moreover, given a measurable set $E$, we define the {\it perimeter of $E$} by 
\[
\PP_{\H(d,\aa)} (E):= \Var_{\H(d,\aa)} (\chi_E).
\] 
See \cite{CDG94, FSS96} for more details on functions with bounded variation for vector fields (including our case) and also \cite{CDPT07} for the special case of the first Heisenberg group $\H^1$.

\begin{theorem}\label{tgns}
On corank $1$ Carnot group $\H(d,\aa)$, there exists a constant $C > 0$ (depending on $\H(d,\aa)$) such that
\begin{align}\label{tarxx}
\|f\|_{\frac{d + 2n + 2}{d + 2n + 1}} \le C \Var_{\H(d,\aa)}(f), \qquad \forall \, f \in \BV(\H(d,\aa)).
\end{align}
\end{theorem}

\begin{proof}
In fact from an approximation argument (see for example \cite[Theorem 2.2.2]{FSS96}) we only need to prove \eqref{tarxx} for $f \in C_0^\infty(\H(d,\aa))$. Note that in this case by integration by parts the right-hand side of \eqref{tarxx} becomes $\|\nabla f\|_1$. Now for such $f$, $1 \le j \le d + 2n$, and $k \in \Z$, we write
\[
F_k := \{(x,t) \in \H(d,\aa) : \, 2^{k - 1} \le |f(x,t)| < 2^k\}.
\]
Then an argument similar to the one in the proof of \cite[Lemma 4.3]{FP22} gives the estimate
\begin{align}\label{estFk}
\m_{d + 2n}(\pi_j(F_k)) \le 2^{-k + 2} \int_{F_{k - 1}} |\X_j f| d\m_{d + 2n + 1}, \qquad \forall \, j, k.
\end{align}
Then by Corollary \ref{c1} and \eqref{estFk} we have
\begin{align}\nonumber
&\int_{\H(d,\aa)} |f|^{\frac{d + 2n + 2}{d + 2n + 1}}  d\m_{d + 2n + 1} < \sum_{k} 2^{\frac{k(d + 2n + 2)}{d + 2n + 1}} \m_{d + 2n + 1}(F_k) \\
\nonumber
\le & \, \CC(d,\aa)\sum_{k} 2^{\frac{k(d + 2n + 2)}{d + 2n + 1}} \prod_{j = 1}^d \m_{d + 2n}(\pi_j(F_k))^{\frac{1}{d + 2n + 1}} \prod_{j = d + 1}^{d + 2n} \m_{d + 2n}(\pi_j(F_k))^{\frac{n + 1}{n(d + 2n + 1)}} \\
\label{inter3}
\le & \, \CC(d,\aa)  2^{\frac{2 (d + 2n + 2)}{d + 2n + 1}} \sum_k \prod_{j = 1}^{d + 2n} \ar_j(k),   
\end{align}
where $ (\ar_j(k))_{k \in \Z}$ is given by
\[
\ar_j(k) := \begin{cases}
\left( \int_{F_{k - 1}} |\X_j f|  d\m_{d + 2n + 1}\right)^{\frac{1}{d + 2n + 1}}, & \mbox{if $1 \le j \le d$} \\
 \left( \int_{F_{k - 1}} |\X_j f|  d\m_{d + 2n + 1}\right)^{\frac{n + 1}{n(d + 2n + 1)}}, &  \mbox{if $d + 1 \le j \le d + 2n$}
\end{cases}.
\]
By H\"older's inequality,
\begin{align}\nonumber
\sum_k \prod_{j = 1}^{d + 2n} \ar_j(k) &\le \prod_{j = 1}^d \left( \sum_k \ar_j(k)^{d + 2n + 1} \right)^{\frac{1}{d + 2n + 1}} \prod_{j = d + 1}^{d + 2n }\left( \sum_k \ar_j(k)^{\frac{2n(d + 2n + 1)}{2n + 1}} \right)^{\frac{2n + 1}{2n(d + 2n + 1)}} \\
\label{inter2}
&\le \prod_{j = 1}^d \left( \sum_k \ar_j(k)^{d + 2n + 1} \right)^{\frac{1}{d + 2n + 1}} \prod_{j = d + 1}^{d + 2n }\left( \sum_k \ar_j(k)^{\frac{n(d + 2n + 1)}{n + 1}} \right)^{\frac{n + 1}{n(d + 2n + 1)}},
\end{align}
where the last ``$\le$'' follows from the embedding  $\ell^{\frac{n(d + 2n + 1)}{n + 1}}(\Z)$ into $\ell^{\frac{2n(d + 2n + 1)}{2n + 1}}(\Z)$ since $\frac{2n}{2n + 1} > \frac{n}{n + 1}$. Then combining \eqref{inter3} with \eqref{inter2} we obtain
\[
\int_{\H(d,\aa)} |f|^{\frac{d + 2n + 2}{d + 2n + 1}}  d\m_{d + 2n + 1} \le \CC(d,\aa)  2^{\frac{2 (d + 2n + 2)}{d + 2n + 1}} \prod_{j = 1}^d \|\X_j f\|_1^{\frac{1}{d + 2n + 1}} \prod_{j = d + 1}^{d + 2n }  \|\X_j f\|_1^{\frac{n + 1}{n(d + 2n + 1)}},
\]
which implies
\begin{align*}
\|f\|_{\frac{d + 2n + 2}{d + 2n + 1}} \le C \prod_{j = 1}^d \|\X_j f\|_1^{\frac{1}{d + 2n + 2}} \prod_{j = d + 1}^{d + 2n }  \|\X_j f\|_1^{\frac{n + 1}{n(d + 2n + 2)}} \le C \|\nabla f\|_1 
\end{align*}
where $C$ is a constant only depending on the underlying group and  the last ``$\le$'' comes from the simple fact that $|\X_j f| \le |\nabla f|$, $1 \le j \le d + 2n$. This proves \eqref{tarxx} and thus this theorem.
\end{proof}

Applying Theorem \ref{tgns} to the set we obtain the corresponding isoperimetric inequality.

\begin{corollary}\label{ciso}
On corank $1$ Carnot group $\H(d,\aa)$, there exists a constant $C > 0$ (depending on $\H(d,\aa)$) such that
\[
\m_{d + 2n + 1}(E)^{\frac{d + 2n + 1}{d + 2n + 2}} \le C \PP_{\H(d,\aa)}(E), \qquad \forall \, E \mbox{ with finite perimeter}.
\]
\end{corollary}

Note that the isoperimetric inequality on the first Heisenberg was first established by Pansu \cite{P82}. Generalizations can be found in \cite{CDG94, GN96}.

\subsection{Generalization to the product spaces}

Similar argument can be applied to the case of product of corank $1$ Carnot groups. In fact, from Proposition \ref{ppro} and Theorem \ref{t1}, we obtain the following theorem without difficulties.

\begin{theorem}\label{tpro}
Assume $\H(d,\aa)$ and $\H(d',\aa')$ are two corank $1$ Carnot groups with $\aa = (\a_1, \ldots, \a_n)$ and $\aa' = (\a_1',\ldots, \a_{n'}')$. On the product space $\H(d,\aa) \times \H(d',\aa')$, using the notations in Proposition \ref{ppro}, there are $d + 2n + d' + 2n'$ projections $\{\bar{\pi}_j\}_{j = 1}^{d + 2n + d'+ 2n'}$ defined in a similar way as \eqref{defbp}. Then it holds that
\begin{align*}
\int_{\H(d,\aa) \times \H(d',\aa')} & \prod_{j = 1}^{d + 2n + d' + 2n'} f_j(\bar{\pi}_j(x,t,x',t')) dxdtdx'dt' \\
\le  \CC(d,\aa,d',\aa') & \prod_{j = 1}^d \|f_j\|_{d + 2n + d' + 2n' + 3} \prod_{j = d + 1}^{d + 2n} \|f_j\|_{\frac{n(d + 2n + d' + 2n' + 3)}{n + 1}}\\
\times & \,  \prod_{j = d + 2n + 1}^{d + 2n + d'} \|f_j\|_{d + 2n + d' + 2n' + 3} \prod_{j = d + 2n + d' + 1}^{d + 2n + d' + 2n'} \|f_j\|_{\frac{n'(d + 2n + d' + 2n' + 3)}{n' + 1}},
\end{align*}
for all non-negative measurable functions $f_1, \ldots, f_{d + 2n + d' + 2n'}$ on $\R^{d + 2n + d' + 2n + 1}$, where 
\begin{align}\label{defCCC}
\CC(d,\aa,d',\aa') := \frac{\|\RR\|_{\frac{3}{2} \to 3}^{\frac{6}{d + 2n + d' + 2n' + 3}}}{\left( \prod_{j = 1}^n \a_j\right)^{\frac{1}{n(d + 2n + d' + 2n'+ 3)}} \left( \prod_{j = 1}^{n'} \a_j'\right)^{\frac{1}{n'(d + 2n + d' + 2n'+ 3)}}}.
\end{align}
\end{theorem}

Generalizations of Theorem \ref{tgns} and Corollary \ref{ciso}, as well as corresponding results for products of three corank $1$ Carnot groups or more are left to the interested reader.

\section*{Acknowledgements}
\setcounter{equation}{0}
YZ would like to thank Prof. Neal Bez for fruitful discussions and bringing \cite{CC09, S11, TW03} to the author's attention 
during the MATRIX-RIMS Tandem Workshop: Geometric Analysis in Harmonic Analysis and PDE. This work was supported by the Research Institute for Mathematical Sciences, an International Joint Usage/Research Center located in Kyoto University.


\mbox{}\\
Ye Zhang\\
Analysis on Metric Spaces Unit  \\
Okinawa Institute of Science and Technology Graduate University \\
1919-1 Tancha, Onna-son, Kunigami-gun \\
Okinawa, 904-0495, Japan \\
E-Mail: zhangye0217@gmail.com \quad or \quad Ye.Zhang2@oist.jp \mbox{}\\

\end{document}